\documentclass{article}

\usepackage{graphicx}
\usepackage{indentfirst}
\usepackage{amsmath,amsfonts,amsthm,amssymb}
\usepackage{mathrsfs}
\usepackage{amscd}

\newcommand{\homo}{{\mathrm H}}
\def\co{\colon\thinspace}
\DeclareMathAlphabet{\mathsfsl}{OT1}{cmss}{m}{sl}

\newtheorem{thm}{Theorem}[section]
\newtheorem{lem}[thm]{Lemma}
\newtheorem{cor}[thm]{Corollary}
\newtheorem{prop}[thm]{Proposition}

\newtheorem*{thm*}{Theorem}

\theoremstyle{definition}
\newtheorem{defn}[thm]{Definition}

\newtheorem{rem}[thm]{Remark}

\newtheorem{qusn}[thm]{Question}

\begin{document}

\title{Dehn surgeries that yield fibred $3$--manifolds}

\author{{\Large Yi Ni}\\American Institute of Mathematics\\
360 Portage Ave, Palo Alto, CA 94306-2244\\and
\\
Department of Mathematics, Columbia University\\ MC 4406, 2990
Broadway, New York, NY 10027}

\maketitle

\begin{abstract}
We study Dehn surgeries on null-homotopic knots that yield fibred
$3$--manifolds when an additional (but natural) homological
restriction is imposed. The major tool used is Gabai's theory of
sutured manifold decomposition. Such surgeries are negative
examples to a question of Michel Boileau. Another result we will
prove is about surgeries which reduce the Thurston norm of a
fibred manifold.
\end{abstract}

\section{Introduction}

One basic question on Dehn surgery is when a Dehn surgery yields a
special type of manifolds. In this paper, we will consider Dehn
surgeries on null-homotopic knots that yield fibred manifolds. All
manifolds we study here will be compact and orientable unless
stated otherwise. Our main result is as follows.

\begin{thm}\label{FibredSurg}
Suppose $Y$ is a compact $3$--manifold with boundary consisting of
tori, $L$ is a nontrivial null-homotopic knot in $Y$. Let $\alpha$
be a nontrivial slope on $T=\partial\mathrm{Nd}(L)$, $X$ is the
manifold obtained from $Y$ by $\alpha$--surgery, $K\subset X$ is
the core of the surgery.

If $X$ fibres over the circle with fibre $F$, such that
\begin{equation}\label{ZeroIntersect}
[K]\cdot[F]=0,\end{equation} then there is an ambient isotopy of
$X$ which takes $K$ to a curve in $F$. Moreover, let $\beta\subset
T$ be the meridian of $L$, then $\beta$ is the frame of $K$
specified by $F\supset K$. Hence $\Delta(\alpha,\beta)=1$, where
$\Delta$ is the distance between two slopes.
\end{thm}

\begin{cor}
Suppose $L$ is a nontrivial null-homotopic knot in a closed
$3$--manifold $Y$, $X$ is obtained by a $\frac pq$--surgery on
$L$, $\frac pq\ne0,\infty$. If $X$ fibres over the circle, then
$\frac pq\in\mathbb Z$, and every integer surgery on $L$ yields a
fibred manifold.
\end{cor}
\begin{proof}
Let $K\in X$ be the core of the surgery.  Since $\frac pq\ne0$,
$K$ is rationally null-homologous in $X$, hence the condition
(\ref{ZeroIntersect}) is satisfied. Let $\alpha,\beta$ be as in
Theorem~\ref{FibredSurg}, then $\Delta(\alpha,\beta)=1$, which
implies that $\frac pq\in\mathbb Z$. Since $\beta$ is the frame of
$K$ specified by $F$, every integer surgery on $L$ yields a fibred
manifold whose monodromy differs from the monodromy of $X$ by a
power of the Dehn twist along $K\subset F$.
\end{proof}

Our theorem is related to the following question of Michel
Boileau.

\begin{qusn}{\bf\cite[Problem~1.80C, Boileau]{Kirby}}\label{Boileau}
Let $L$ be a null-homotopic knot in a closed orientable
$3$--manifold $Y$. If a nontrivial surgery on $L$ yields a
manifold that fibres over the circle, does it follow that $L$ is a
fibred knot and the surgery is longitudinal?
\end{qusn}

This question was answered affirmatively in the case that $Y=S^3$
by Gabai \cite{G3}. Boileau and Wang showed that if the surgery is
fibred then either the surgery is longitudinal or $Y$ itself is
fibred \cite{BW}. In \cite{Ni} the question was answered
affirmatively for {\it null-homologous} knots of genus $>1$ in
$L$--spaces, using Heegaard Floer homology.

However, we note that there are simple examples of surgeries on
null-homotopic knots yielding fibred manifolds and satisfying the
homological restriction~(\ref{ZeroIntersect}), hence they are
negative examples to Question~\ref{Boileau}. A construction of
such examples can be given as follows. Take any unknotting number
one fibred knot $k\subset S^3$. There exists a circle
$\gamma\subset S^3-k$ which has linking number zero with $k$ such
that a $\pm1$ surgery on $\gamma$ yields a solid torus, which
means that there exists a winding number zero knot $L$ in the
solid torus $U$, such that a surgery on $L$ yields the fibred
manifold $S^3-k$. In order to construct knots in closed manifolds,
one can take any closed manifold $Y$ which fibres over the circle,
embed $U$ into $Y$ such that the core of $U$ is transverse to the
fibres, then $L\subset Y$ is a knot that satisfies the conditions
in Theorem~\ref{FibredSurg}.



The proof of Theorem~\ref{FibredSurg} uses Gabai's theory of taut
foliations and sutured manifold decomposition. Sutured manifold
theory has been successfully applied to study surgery on
null-homotopic knots by Lackenby \cite{La}. In addition to the use
of \cite{G2} as in \cite{La0,La} we borrow some ideas from
\cite{G3} and \cite{Gh}, which have been used to show that knot
Floer homology detects fibred knots \cite{Gh,Ni}.

The same argument can be used to prove the following theorem.

\begin{thm}\label{NormDecr}
Suppose $M$ is a compact $3$--manifold with boundary consisting of
tori, $T$ is a component of $\partial M$, $\alpha,\beta$ are two
different slopes on $T$, $X,Y$ are the manifolds obtained by
filling $T$ along $\alpha,\beta$, $K\subset X$ is the core of the
$\alpha$--filling. Let $i_X\co\homo_2(M,\partial
M-T)\to\homo_2(X,\partial X), i_Y\co\homo_2(M,\partial
M-T)\to\homo_2(Y,\partial Y)$ be the maps on homology induced by
inclusions.

Suppose $X$ fibres over the circle with fibre $F$ such that there
exists a $\theta\in\homo_2(M,\partial M-T)$ satisfying
$i_X(\theta)=[F]$. If the Thurston norm of $i_Y(\theta)$ is less
than the Thurston norm of $F$, then there is an ambient isotopy of
$X$ which takes $K$ to a curve in $F$. Moreover, $\beta$ coincides
with the frame on $K$ which is specified by the surface $F$. Hence
$\Delta(\alpha,\beta)=1$.
\end{thm}

\begin{rem}
The conclusion $\Delta(\alpha,\beta)=1$ in the above theorems can
also be proved by the argument in \cite[Theorem~5.2]{La0}. The
conclusion of Theorem~\ref{NormDecr} has also been obtained by
John Luecke \cite{Lu} in the case when $Y$ is a solid torus. Our
argument in this paper is closer to Lackenby's, while Luecke's
argument involves the combinatorial techniques from \cite{GL}.
\end{rem}

\begin{rem}
In Theorem~\ref{NormDecr}, it is interesting to ask what happens
when $X$ is not fibred. One may naturally guess that $K$ can be
isotoped to a curve in some taut surface representing
$i_X(\theta)$, but this picture is not correct: Gabai told the
author a method of constructing negative examples. In spite of
this disappointing answer, the above guess is true in the special
case when $X$ and $Y$ are link complements related by a crossing
change, according to a result due to Scharlemann and Thompson
\cite[Proposition~3.1]{SchTh}.
\end{rem}

This work is motivated by works in Heegaard Floer homology
\cite{Gh,Ni}, but the argument here is quite classical, we do not
need Heegaard Floer homology and contact topology at all. The only
gauge theoretical aspect in this paper is the citation of Property
P \cite{KM}, which can be replaced by the Knot Complement Theorem
\cite{GL} if we assume there are no fake $3$--cells in $Y$. One
should even be able to eliminate foliations as in
\cite{Sch,La0,La}.

The paper is organized as follows. In Section~2 we will give some
preliminaries on sutured manifold decompositions. In Section~3 we
study knots in product manifolds via a method of Gabai \cite{G3}.
The key result in this section is Proposition~\ref{NoOrth}. The
proofs of the above two theorems are just its routine
applications, which are given in Section 4 and 5.

\

\noindent{\bf Acknowledgements.}\quad We are grateful to Michel
Boileau, David Gabai and Tao Li for some very interesting
discussions, to Marc Lackenby for helpful comments, and to John
Luecke for a detailed description of his earlier proof of a
special case of Theorem~\ref{NormDecr}. The author is partially
supported by an AIM Five-Year Fellowship. This research was
partially conducted during the period the author was employed by
the Clay Mathematics Institute as a Liftoff Fellow, and when the
author visited University of Minnesota, Twin Cities. The author
wishes to thank the above institutions for their supports, and
special thanks are due to Tian-Jun Li for his hospitality.

\section{Preliminaries on sutured manifolds}

Sutured manifold decomposition was introduced by Gabai in
\cite{G1} in order to construct taut foliations. In this section,
we will briefly review some basic definitions about sutured
manifolds, then discuss the main result in \cite{G2}.

\begin{defn}
A {\it sutured manifold} $(M,\gamma)$ is a compact oriented
3--manifold $M$ together with a set $\gamma\subset \partial M$ of
pairwise disjoint annuli $A(\gamma)$ and tori $T(\gamma)$. The
core of each component of $A(\gamma)$ is a {\it suture}, and the
set of sutures is denoted by $s(\gamma)$.

Every component of $R(\gamma)=\partial M-\mathrm{int}(\gamma)$ is
oriented. Define $R_+(\gamma)$ (or $R_-(\gamma)$) to be the union
of those components of $R(\gamma)$ whose normal vectors point out
of (or into) $M$. The orientations on $R(\gamma)$ must be coherent
with respect to $s(\gamma)$, hence every component of $A(\gamma)$
lies between a component of $R_+(\gamma)$ and a component of
$R_-(\gamma)$.
\end{defn}

\begin{defn}
Let $S$ be a compact oriented surface with connected components
$S_1,\dots,S_n$. We define
$$x(S)=\sum_i\max\{0,-\chi(S_i)\}.$$
Let $M$ be a compact oriented 3--manifold, $A$ be a compact
codimension--0 submanifold of $\partial M$. Let
$h\in\homo_2(M,A)$. The {\it Thurston norm} $x(h)$ of $h$ is
defined to be the minimal value of $x(S)$, where $S$ runs over all
the properly embedded surfaces in $M$ with $\partial S\subset A$
and $[S]=h$.
\end{defn}

\begin{defn}
A properly embedded surface $S\subset M$ is {\it taut}, if $S$ is
incompressible and Thurston norm minimizing in $\homo_2(M,\partial
S)$. A sutured manifold $(M,\gamma)$ is {\it taut}, if $M$ is
irreducible, and $R(\gamma)$ is taut.
\end{defn}

\begin{defn}
Let $(M,\gamma)$ be a sutured manifold, and $S$ a properly
embedded surface in M, such that no component of $\partial S$
bounds a disk in $R(\gamma)$ and no component of $S$ is a disk
with boundary in $R(\gamma)$. Suppose that for every component
$\lambda$ of $S\cap\gamma$, one of 1)--3) holds:

1) $\lambda$ is a properly embedded non-separating arc in
$\gamma$.

2) $\lambda$ is a simple closed curve in an annular component $A$
of $\gamma$ in the same homology class as $A\cap s(\gamma)$.

3) $\lambda$ is a homotopically nontrivial curve in a toral
component $T$ of $\gamma$, and if $\delta$ is another component of
$T\cap S$, then $\lambda$ and $\delta$ represent the same homology
class in $\homo_1(T)$.

Then $S$ is called a {\it decomposition surface}, and $S$ defines
a {\it sutured manifold decomposition}
$$(M,\gamma)\stackrel{S}{\rightsquigarrow}(M',\gamma'),$$
where $M'=M-\mathrm{int}(\mathrm{Nd}(S))$ and
\begin{eqnarray*}
\gamma'\;\;&=&(\gamma\cap M')\cup \mathrm{Nd}(S'_+\cap
R_-(\gamma))\cup
\mathrm{Nd}(S'_-\cap R_+(\gamma)),\\
R_+(\gamma')&=&((R_+(\gamma)\cap M')\cup S'_+)-\mathrm{int}(\gamma'),\\
R_-(\gamma')&=&((R_-(\gamma)\cap M')\cup
S'_-)-\mathrm{int}(\gamma'),
\end{eqnarray*}
where $S'_+$ ($S'_-$) is that component of
$\partial\mathrm{Nd}(S)\cap M'$ whose normal vector points out of
(into) $M'$.
\end{defn}

\begin{defn}
A decomposition surface is called a {\it product disk}, if it is a
disk which intersects $s(\gamma)$ in exactly two points. A
decomposition surface is called a {\it product annulus}, if it is
an annulus with one boundary component in $R_+(\gamma)$, and the
other boundary component in $R_-(\gamma)$.
\end{defn}

\begin{defn}
An {\it I-cobordism} between closed connected surfaces $T_0$ and
$T_1$ is a compact $3$--manifold $V$ such that $\partial V=T_0\cup
T_1$ and for $i=0,1$ the induced maps
$j_i\co\homo_1(T_i)\to\homo_1(V)$ are injective.
\end{defn}

It is noted in \cite[Lemma 1.5]{G2} that for an I-cobordism $V$
between $T_0$ and $T_1$, the maps $j_1,j_2$ induce isomorphisms on
$\homo_1(\cdot;\mathbb Q)$. One then easily sees that $V$ has the
same rational homology type as $T_0\times I$.

\begin{defn}
Let $M$ be a compact $3$--manifold, $S$ a properly embedded
surface in $M$, and $T$ a toral component of $\partial M$ such
that $T\cap S=\emptyset$. $M$ is {\it $S_T$--atoroidal} if
boundary parallel tori are the only surfaces which are I-cobordant
to $T$ by cobordisms contained in $M-S$. If the boundary component
$T$ is understood, then we say that $M$ is $S$--atoroidal.
\end{defn}

The main result in \cite{G2} is as follows.

\begin{thm}[Gabai]\label{Gabai}
Let $M$ be a compact irreducible $3$--manifold whose boundary
consists of tori. $T$ is a component of $\partial M$ and $S$ is a
taut surface representing a nontrivial element in
$\homo_2(M,\partial M-T)$, $S\cap T=\emptyset$. If $M$ is
$S_T$--atoroidal, then except at most one slope the manifold $N$
obtained by filling $M$ along a slope in $T$ possesses a taut
foliation $\mathscr F$ such that $S$ is a compact leaf of
$\mathscr F$, and the core $C$ of the filling is transverse to
$\mathscr F$, hence $C$ is of infinite order in $\pi_1(N)$.
\end{thm}
\begin{proof}[Sketch of proof]
There exists a sequence
$$(M,\partial M)=(M_0,\gamma_0)\stackrel{S=S_1}{\rightsquigarrow}
(M_1,\gamma_1)\stackrel{S_2}{\rightsquigarrow}\cdots\stackrel{S_n}{\rightsquigarrow}(M_n,\gamma_n)$$
of sutured manifold decompositions with the following properties:

1) Each $(M_i,\gamma_i)$ is taut and each separating component of
$S_{i+1}$ is a product disk.

2) All $S_i$'s are disjoint from $T$. (Hence some component of
$\gamma_n$ is the torus $T$.)

3) $(M_n,\gamma_n)$ is a union of a product sutured manifold and a
sutured manifold $(H,\delta)$, where $H=T^2\times I$,
$T=T^2\times0$, $\delta\cap(T^2\times1)\ne\emptyset$.

The idea is to inductively construct sutured manifold
decompositions satisfying 1) and 2) until such construction can no
longer be done. Now the last sutured manifold $(M_n,\gamma_n)$
should be a union of a product sutured manifold and a sutured
manifold $(H,\delta)$, where $H$ is an I-cobordism between $T$ and
another torus $P\subset M-S$. Since $M$ is $S_T$--atoroidal, $H$
must be $T^2\times I$.

Fix a slope on $T$, we fill each $T\subset
\partial M_i$ along this slope by a solid torus, to get the
sequence
$$\mathcal N:\quad(N,\partial M-T)=(N_0,\delta_0)\stackrel{S_1}{\rightsquigarrow}(N_1,\delta_1)\stackrel{S_2}{\rightsquigarrow}
\cdots\stackrel{S_n}{\rightsquigarrow}(N_n,\delta_n).$$ By 3), the
component $\widehat H$ of $N_n$ containing $T$ satisfies $\widehat
H=D^2\times S^1$ and $s(\delta_n)\cap\widehat H$ is a union of
$2r\;(\ne0)$ parallel essential simple closed curves in $\partial
D^2\times S^1$. If the slope on $T$ is not the one that kills
$s(\delta_n)\cap\widehat H$ in $\pi_1(\widehat H)$, then one can
decompose $(N_n,\delta_n)$ along a $D^2\times\mathrm{point}$ to
get a product sutured manifold. Hence the above sequence $\mathcal
N$ is extended to a sutured manifold hierarchy.

Now we can apply \cite[Theorem~5.1]{G1} to the sutured manifold
hierarchy gotten in the last paragraph to obtain the desired
foliations.
\end{proof}

\section{Knots in product manifolds}\label{SectProd}

In this section we will study knots in product manifolds. Let $F$
be a compact surface, $K\subset F\times I$ is a knot which is not
contained in a $3$--ball, (hence $F\ne D^2\;\text{\rm or}\;S^2$,)
$M_1=F\times I-\mathrm{int}(\mathrm{Nd}(K))$, $T=\partial Nd(K)$.

\begin{lem}\label{BoundSolTor}
Suppose $R\subset M_1$ is a torus which is I-cobordant to $T$,
then $R$ bounds a solid torus $U$ in $F\times I$, $K\subset U$.
\end{lem}
\begin{proof}
Let $V$ be the I-cobordism between $T$ and $R$. If $R$ is
incompressible in $F\times I$, then $R$ is isotopic to $F\times
t$, which is impossible since $[R]=0$ in $\homo_2(F\times I)$.

Now $R$ is compressible in $F\times I$, let $S$ be the sphere
obtained by compressing $R$, then $R$ is obtained by adding a tube
to $S$. $S$ bounds a ball $B$ in $F\times I$.

If the tube is contained in $B$, then $R$ bounds a
cube-with-knotted-hole $U\subset B$ in $F\times I$. $K\subset
F\times I-U$ since $K\not\subset B$. Now $M_1=U\cup V$ has only
one boundary component $T$, which contradicts to the fact that
$M_1=F\times I-\mathrm{int}(\mathrm{Nd}(K))$.

If the tube is not contained in $B$, then $R$ bounds a solid torus
$U$. The same argument as in the last paragraph shows $K\subset
U$.
\end{proof}

\begin{defn}
Suppose $E$ is a compact subsurface of a compact surface $F$. $E$
is {\it essential} if no component of
$\mathrm{Fr}(E)=E\cap\overline{F-E}$ is a circle that bounds a
disk in $F$ or an proper arc that cobounds a disk in $F$ with an
arc in $\partial F$.
\end{defn}

Let $E\times I\subset M_1$ be the {\it characteristic product
pair}
 of $M_1$. Namely, $E$ is a maximal (up to isotopy) compact
essential subsurface of $F$, such that $K$ can be isotoped in
$F\times I$ to be disjoint from $E\times I$. Let
$M_2=\overline{M_1-E\times I}$, $G=\overline{F-E}$. Now $K$ is a
knot in $G\times I$. By the choice of $E$, its complement $G$
should be connected. Let $\gamma_1=(\partial F\times I)\cup T$,
$\gamma_2=(\partial G\times I)\cup T$, then $(M_1,\gamma_1),
(M_2,\gamma_2)$ are sutured manifolds.

\begin{defn}
Suppose $S\subset M_2$ is a non-separating decomposition surface
which gives a taut decomposition of $(M_2,\gamma_2)$, $S\cap
T=\emptyset$. $S$ is {\it tautly extendable} if $S$ also gives a
taut decomposition of $(G\times I,\partial G\times I)$. $M_2$ has
the {\it taut-extension property}, if every non-separating
decomposition surface $S\subset M_2-T$ which gives a taut
decomposition of $(M_2,\gamma_2)$ is tautly extendable.
\end{defn}

\begin{prop}\label{NoOrth} Suppose $M_2$ has the taut-extension
property. The inclusion $K\subset G\times I$ induces a map
$$i_*\co\homo_1(K;\mathbb Q)\to\homo_1(G;\mathbb Q).$$
If $\sigma$ is a nonzero element in $\homo_1(G,\partial G;\mathbb
Q)$, then $\sigma\cdot i_*([K])\ne0$.
\end{prop}

Otherwise there exists a non-separating simple oriented curve
(which is a circle or a proper arc) $C\subset G$ such that
$[C]\cdot i_{*}([K])=0$.

{\noindent\bf Case 1.} The curve $C$ is a proper arc with ends on
different components of $\partial G$.

Suppose $\sigma,\tau$ are the two components of $\partial G$ that
contain $\partial C$, $a\in\sigma-\partial C$ is a point. For a
proper surface $S\subset M_2$, let $\partial_i(S)=S\cap (G\times
i)$, $i=0,1$, $\partial_v(S)=S\cap(\partial G\times I)$.

Let $\mathcal S_m(+C)$ be the set of properly embedded oriented
surfaces $S\subset G\times I$, such that $S\cap K=\emptyset$,
$\partial_0 S=C\times0$, $\partial_1 S=-C\times1$, and the
algebraic intersection number between $S$ and $a\times I$ is $m$.
Here $-C$ denotes the same curve $C$, but with opposite
orientation. Similarly, let $\mathcal S_m(-C)$ be the set of
properly embedded surfaces $S\subset G\times I$, such that $S\cap
K=\emptyset$, $\partial_0 S=-C\times0$, $\partial_1 S=C\times1$,
and the algebraic intersection number of $S$ with $a\times I$ is
$m$. Since $[C]\cdot i_{*}([K])=0$, $\mathcal S_{m}(\pm
C)\ne\emptyset$.

\begin{lem}\label{ExistTaut}
When $m$ is sufficiently large, there exists a connected surface
$S\subset\mathcal S_m(+C)$ such that $S$ gives a taut
decomposition of $M_2$. The same statement holds for $S_m(-C)$.
\end{lem}

This lemma is implicitly contained in \cite[Theorem~3.13]{G1}, and
the details are given in \cite[Lemma~6.4]{Ni}.

Suppose $S\subset M_2$ is a properly embedded surface which is
transverse to $\partial G\times0$. For any component $S_0$ of $S$,
we define
$$y(S_0)=\max\{\frac{|S_0\cap(\partial G\times0)|}2-\chi(G),0\},$$
and let $y(S)$ be the sum of $y(S_0)$ with $S_0$ running over all
components of $S$. Let $y(\mathcal S_m(\pm C))$ be the minimal
value of $y(S)$ for all $S\in\mathcal S_m(\pm C)$. If
$S\in\mathcal S_m(\pm C)$, let $S'$ be the surface obtained by
doing oriented cut-and-paste to $S$ and $G\times1$, it is obvious
that $S'\in\mathcal S_{m+1}(\pm C)$ and $y(S')=y(S)+y(G)$. Hence
we have
$$y(\mathcal S_{m+1}(\pm C))\leq y(\mathcal S_{m}(\pm C))+
y(G).$$

The following key lemma is essentially \cite[Lemma~6.5]{Ni}, the
argument in the proof is due to Gabai \cite{G3}.

\begin{lem}\label{SumLarge}
For any positive integers $p,q$,
$$y(\mathcal S_p(+C))+y(\mathcal
S_q(-C))>(p+q)y(G).$$
\end{lem}

Suppose $S_1\in\mathcal S_p(+C),S_2\in\mathcal S_q(-C)$, $p,q>0$,
and $y(S_1)=y(\mathcal S_p(+C))$, $y(S_2)=y(\mathcal S_q(-C))$.
Isotope $S_1,S_2$ so that they are transverse, and
$|(\partial_vS_1)\cap (\partial_vS_2)|$ is minimal. The following
lemma is obvious.

\begin{lem}
On $\sigma\times I$, $\partial_vS_1$ and $\partial_vS_2$ have
exactly $p+q+1$ intersection points, and their orientations are
the same. The same statement holds for $\tau\times I$.
\end{lem}

Now $S_1\cap S_2$ consists of some circles and exactly $p+q+1$
arcs, each arc has one end on $\sigma\times I$ and the other end
on $\tau\times I$. Note that two arcs among them are
$C\times\{0,1\}$.

Perform oriented cut-and-paste to $S_1,S_2$, we get a proper
surface $P$, then we isotope $P$ slightly such that it lies in
$\mathrm{int}(M_2)$. It is easy to show that
$\chi(P)=\chi(S_1)+\chi(S_2)-2$, hence
\begin{equation}\label{yAdditive}
y(P)=y(S_1)+y(S_2).
\end{equation}

\begin{defn}
A properly embedded surface in $M_2$ is {\it boring}, if its Euler
characteristic is nonnegative, and its algebraic intersection
number with $a\times I$ is $0$.
\end{defn}

{\noindent\bf Claim 0.} There is exactly one component of $S_1$
whose intersection with $G\times\{0,1\}$ is nonempty. Moreover,
this component is not a disk or annulus.

Since $S_1\cap (G\times\{0,1\})=C\times\{0,1\}$, the component of
$S_1$ which contain $C\times0$ must also contain $C\times1$. The
second statement holds since $M_2$ contains no nontrivial product
disks or product annuli.

{\noindent\bf Claim 1.} We can assume that no component of
$\partial S_1,\partial S_2$ is the boundary of a disk in $\partial
G\times I$. Moreover, we can assume that $S_1,S_2$ contain no
boring components.

If one component of $\partial S_1$ is the boundary of a disk in
$\partial G\times I$, without loss of generality we can assume no
other components of $\partial S_1,\partial S_2$ are contained in
the disk, then we can cap off this component of $\partial S_1$ by
the disk to get a new surface $S'_1\subset\mathcal S_p(+C)$,
$y(S'_1)\le y(S_1)$. This proves the first statement.

Suppose $B$ is a boring component of $S_1$. We can remove $B$
without increasing $y(S_1)$, and the new surface is still
contained in $\mathcal S_p(+C)$.

{\noindent\bf Claim 2.} We can assume that no component of
$S_1\cap S_2$ bounds a disk in $S_1$ or $S_2$, hence no component
of $P$ is a boring sphere or disk.

If a component of $S_1\cap S_2$ bounds a disk in $S_1$, then this
component also bounds a disk in $S_2$ since $S_2$ is
incompressible. Since $M_2$ is irreducible, we can isotope $S_1$
to eliminate the components of $S_1\cap S_2$ that bound disks in
$S_1$ or $S_2$. If a component $Q$ of $P$ is a boring sphere or
disk, then $Q$ is a component of $S_1$ or $S_2$, since $S_1\cap
S_2$ contains no circle that bounds a disk in $S_1$ or $S_2$. Now
we apply Claim 1 to get a contradiction.

{\noindent\bf Claim 3.} We can assume that there is no subsurface
$Q$ of $P$, such that $Q$ is the union of some components of $P$,
$Q\cdot(a\times I)=0$, and $\chi(Q)=0$.

Suppose $Q$ is such a subsurface of $P$, by Claims 1 and 2 $Q$ is
the union of two collections of annuli or tori
$A_1,A_2,\dots,A_{m}$ and $B_1,\dots,B_n$, where $A_i\subset S_1$,
$B_j\subset S_2$. Let
\begin{eqnarray*}
S_1'&=&(S_1-\bigcup_{i=1}^m A_i)\cup \bigcup_{j=1}^n (-B_{j})\\
S_2'&=&(S_2-\bigcup_{j=1}^n B_{j})\cup \bigcup_{i=1}^m (-A_{i}).
\end{eqnarray*}
Here $-A_i,-B_j$ means $A_i,B_j$ with opposite orientation.

If the surface $A_1$ is a component of $S_1$, by Claim~1 we have
$A_1\cdot(a\times I)\ne0$, then $A_1$ would separate $G\times0$
from $G\times1$, which contradicts to Claim~0.

Now $S_2\cap A_1\ne\emptyset$, a small isotopy will arrange that
$|S_1'\cap S_2'|<|S_1\cap S_2|$. Moreover, $y(S_1')=y(S_1)$,
$y(S_2')=y(S_2)$. We want to show that $S_1'\in\mathcal S_p(+C)$,
$S_2'\in\mathcal S_q(-C)$. Obviously, $\partial_0 S_1'=\partial_0
S_1=C\times 0$, $\partial_1 S_1'=\partial_1 S_1=-C\times 1$.
Moreover, $S_1'\cdot (a\times I) =S_1\cdot (a\times I)$ since
$Q\cdot(a\times I)=0$. Thus $S_1'\in\mathcal S_p(+C)$. Similarly,
$S_2'\in\mathcal S_q(-C)$. Therefore, we can replace $S_1,S_2$
with $S_1',S_2'$, then continue our argument.

Now we are in a position to prove Lemma~\ref{SumLarge}.

\begin{proof}[Proof of Lemma~\ref{SumLarge}]
Suppose $y(\mathcal S_p(+C))+y(\mathcal S_q(-C))\le(p+q)y(G)$. Let
$S_1,S_2$ be as above, and suppose they satisfy Claims 1--3.
Define a function $$\varphi\co (G\times I-P)\to\mathbb Z$$ as
follows. When $z\in G\times0$, $\varphi(z)=0$. In general, given
$z\in G\times I-P$, choose a path from $G\times0$ to $z$,
$\varphi$ is defined to be the algebraic intersection number of
this path with $P$.

Any closed curve in $G\times I$ should have zero algebraic
intersection number with any proper surface in
$G\times\mathrm{int}(I)$, thus $\varphi$ is well-defined.
Moreover, the value of $\varphi$ on $G\times1$ is $p+q$.

Let $J_i$ be the closure of $\{x\in (G\times I
-P)|\:\varphi(x)=i\}$, $P_i=J_{i-1}\cap J_i$. Thus
$P=\sqcup_{i=1}^{m}P_i$ for some $m\ge p+q$, and
$\cup_{k=0}^{i-1}J_k$ gives a homology between $G\times0$ and
$P_i$. $P$ is homologous to $(p+q)G$ in $G\times I$, $G\times0$ is
Thurston norm minimizing in $G\times I$,
$y(P)=y(S_1)+y(S_2)\le(p+q)y(G)$, so we must have $y(P_i)=y(G)$
for each $i$, and $m=p+q$ except possibly when $\chi(G)=0$. By
Claims~1--3, we conclude that $m=p+q$ and $P_i$ is parallel to
$G\times0$ in $G\times I$.

Suppose $K\subset J_r$, then $J_r-\mathrm{int}(\mathrm{Nd}(K))$ is
homeomorphic to $M_2$. Since $P$ is gotten by doing cut-and-paste
to $S_1,S_2$, we can isotope $S_1$ so that $S_1\cap J_i$ consists
of product annuli and disks. We denote $S_1\cap J_i$ by $C_i\times
I$, where $C_i$ is the collection of some curves in $P_i$.
Obviously, $[C_i]$ is homologous to $[C]$ in $\homo_2(G,\partial
G)$. Since $[C]\ne0$, at least one component of $C_i$ is
homologically nontrivial, which implies that
$J_r-\mathrm{int}(\mathrm{Nd}(K))=M_2$ contains a nontrivial
product disk or annulus, a contradiction.
\end{proof}

\begin{proof}[Proof of Proposition~\ref{NoOrth} in Case 1]
By Lemma~\ref{ExistTaut}, when $m$ is large there exist
$S_1\in\mathcal S_m(+C)$, $S_2\in\mathcal S_m(-C)$, such that they
give taut decompositions of $M_2$. By the taut-extension property,
$S_1,S_2$ also give taut decompositions of $G\times I$. Gabai's
work in \cite[Section~5]{G1} then implies that there exist two
taut foliations $\mathscr F_1,\mathscr F_2$ of $G\times I$, such
that $G\times\{0,1\}$ are compact leaves of the foliations, and
$\mathscr F_1,\mathscr F_2$ are transverse to $\partial G\times
I$.

Glue $G\times 0$ to $G\times 1$ by the identity, we get two taut
foliations $\mathscr F'_1,\mathscr F_2'$ of $G\times S^1$. The two
surfaces $S_1,S_2$ are glued to two surfaces $S'_1,S'_2$. We have
$\chi(S'_i)=\chi(S_i)-1=-y(S_i)$.

Let $e(\mathscr F)$ be the Euler class of a foliation $\mathscr
F$. As in the proof of \cite[Theorem~1.4]{Gh}, we have
\begin{eqnarray*}
\chi(S'_1)&=\langle e(\mathscr F_1'),[S_1']\rangle&=\langle e(\mathscr F_1'),[C\times S^1]\rangle+m\chi(G),\\
\chi(S'_2)&=\langle e(\mathscr F_2'),[S_2']\rangle&=\langle
e(\mathscr F_2'),-[C\times S^1]\rangle+m\chi(G).
\end{eqnarray*}

By Lemma~\ref{SumLarge}, $\langle e(\mathscr F_1'),[C\times
S^1]\rangle+\langle e(\mathscr F_2'),-[C\times S^1]\rangle<0$,
hence one summand on the left hand side is nonzero, which
contradicts to \cite[Corollary~1]{T}.
\end{proof}

{\noindent\bf Case 2.} The curve $C$ is a circle or an arc with
ends in the same component of $\partial G$.

If $C$ is an arc with ends in the same component of $\partial G$,
we can connect the two ends by an arc in $\partial G$ to get a
closed curve $\widehat C$. $\widehat C$ is homologous to $C$ in
$\homo_1(G,\partial G)$, so we can just work with $\widehat C$.
The proof when $C$ is a circle is essentially the same as in Case
1, it is even slightly simpler at some technical points. (For
example, we can just work with the Thurston norm $x$, and do not
need its modification $y$.) We will not give the details of the
proof.

\section{Surgery on null-homotopic knots}

In this section, we are going to prove Theorem~\ref{FibredSurg}.
The notation is as in the statement of Theorem~\ref{FibredSurg}.

\begin{prop}{\bf(Boileau--Wang, \cite[Proposition~3.2]{BW})}\label{DegOne}
Suppose $P$ is a compact $3$--manifold, and $k$ is a
null-homotopic knot in $P$. If $Q$ is obtained by Dehn surgery on
$k$, then there is a proper degree--$1$ map from $Q$ to $P$. Let
$k'\subset Q$ be the core of the surgery, then the map can be
chosen such that its restriction to $Q-\mathrm{Nd}(k')$ is a
homeomorphism onto $P-\mathrm{Nd}(k)$.
\end{prop}

Here a map $f\co Q\to P$ is {\it proper} if $f^{-1}(\partial
P)=\partial Q$. Note that in Boileau--Wang's original paper the
result is stated for closed irreducible $3$--manifolds, but the
extra conditions are not necessary for the proof.

\begin{lem}\label{IrrCompl}
Let $M=Y-\mathrm{int}(\mathrm{Nd}(L))$ be the exterior of the
knot, then $M$ is irreducible.
\end{lem}
\begin{proof}
We first consider the case that $X\ne S^2\times S^1$, thus $X$ is
irreducible. If $S$ is an essential sphere in $M$, then $S$ bounds
a $3$--ball $B$ in $X$, and $B\supset K$. Hence $S$
 bounds a compact $3$--manifold $B'$ in $Y$, such that $L$ is a
 null-homotopic knot in $B'$, and a nontrivial surgery on $L$
 yields $B$.

By Proposition~\ref{DegOne}, there is a proper degree--$1$ map
from the ball $B$ to $B'$, hence $B'$ is a homotopy $3$--cell (see
\cite[Lemma~15.12]{Hem}). In other words, a nontrivial surgery on
$K\subset B$ yields a homotopy $3$--cell. Now Property P \cite{KM}
implies that $K$ is the unknot in $B$, so $L$ is the unknot in
$B'$, a contradiction.

Now consider the case that $X=S^2\times S^1$. If $S$ is a
separating essential sphere in $M$, one can get contradiction by
the same argument as before. Now suppose $S\subset M$ is a
nonseparating sphere. Let $N$ (or $N'$) be the compact manifold
obtained by cutting $X$ (or $Y$) open along $S$, $\widehat N$ (or
$\widehat{N'}$) be the closed manifold obtained by capping off the
2 sphere boundary components of $N$ (or $N'$) by balls. Now $L$
can be viewed as a nontrivial null-homotopic knot in $\widehat
{N'}$, such that a nontrivial surgery on $L$ yields
$\widehat{N}=S^3$. Using Property P, we can rule out this case as
before.
\end{proof}

\begin{lem}\label{DisjFibre} $K$ can be isotoped to be disjoint from
$F$.
\end{lem}
\begin{proof}
By the homological restriction (\ref{ZeroIntersect}), there exists
a unique element $$\theta\in\homo_2(M,\partial M-T)$$ whose image
in $\homo_2(X,\partial X)$ is $[F]$. Suppose $(F',\partial
F')\subset(M,\partial M-T)$ is a taut surface in the homology
class $\theta$. Since $L$ is null-homotopic in $Y$, by
\cite[Theorem~A.21]{La} $F'$ is taut in $X$, hence $F'$ is
isotopic to the fibre $F$.
\end{proof}

\begin{lem}
$X\ne S^2\times S^1$.
\end{lem}
\begin{proof}
Otherwise $M$ is reducible by Lemma~\ref{DisjFibre}, which
violates Lemma~\ref{IrrCompl}.
\end{proof}

\begin{prop}[Boileau--Wang]\label{YFibred}
Let $\theta$ be the homology class in the proof of
Lemma~\ref{DisjFibre}, $i_Y(\theta)$ is its image in
$\homo_2(Y,\partial Y)$. Then $Y$ is the connected sum of $Y^*$
and a homotopy $3$--sphere, where $Y^*$ fibres over the circle,
and the fibre of $Y^*$ represents the homology class
$i_Y(\theta)$.
\end{prop}
\begin{proof}
Let $p\co\pi_1(Y)\to\mathbb Z$ be the homomorphism dual to
$i_Y(\theta)$. By Proposition~\ref{DegOne} and the proof of
\cite[Theorem~2.1]{BW}, $\mathrm{Ker}\;p$ is finitely generated.
The result then follows from Stallings' Fibration Theorem
\cite{St}.
\end{proof}

Cut $X$ open along $F$, we get a product $F\times I$, thus $K$ is
a knot in $F\times I$. Let $(M_1,\gamma_1),(M_2,\gamma_2),E,G$ be
as in Section~\ref{SectProd}. We can decompose $(M_1,\gamma_1)$
along a collection $\mathcal C\times I$ of non-separating product
disks and annuli to get $(M_2,\gamma_2)$.

\begin{prop}\label{DecomExt}
If $M$ is $F_T$--atoroidal, then $M_2$ has the taut-extension
property.
\end{prop}
\begin{proof}
Suppose $S$ is a non-separating decomposition surface in the
sutured manifold $(M_2,\gamma_2)$ such that $S\cap T=\emptyset$
and the decomposition
$$(M_2,\gamma_2)\stackrel{S}{\rightsquigarrow}(M_3,\gamma_3)$$
yields a taut sutured manifold,

As in the proof of Theorem~\ref{Gabai}, we can extend the taut
decomposition sequence
$$(M,\partial M)=(M_0,\gamma_0)\stackrel{F}{\rightsquigarrow}(M_1,\gamma_1)\stackrel{\mathcal C\times I}
{\rightsquigarrow}(M_2,\gamma_2)\stackrel{S}{\rightsquigarrow}(M_3,\gamma_3)$$
to a sequence
$$(M_0,\gamma_0)\stackrel{F}{\rightsquigarrow}(M_1,\gamma_1)\stackrel{\mathcal C\times I}
{\rightsquigarrow}(M_2,\gamma_2)\stackrel{S}{\rightsquigarrow}(M_3,\gamma_3)\stackrel{S_4}{\rightsquigarrow}
\cdots\stackrel{S_n}{\rightsquigarrow}(M_n,\gamma_n)$$ with the
properties 1),2),3) there.

Fix a slope on $T$, then we can fill each $T\subset M_i$ along
this slope by a solid torus to get the sequence
$$\mathcal N:(N_0,\delta_0)\stackrel{F}{\rightsquigarrow}(N_1,\delta_1)\stackrel{\mathcal C\times I}{\rightsquigarrow}
\cdots\stackrel{S_n}{\rightsquigarrow}(N_n,\delta_n).$$ As argued
in Theorem~\ref{Gabai}, if the slope on $T$ is not the one that
kills $s(\delta_n)\cap\widehat H$ in $\pi_1(\widehat H)$, then one
can decompose $(N_n,\delta_n)$ along a $D^2\times\mathrm{point}$
to get a product sutured manifold. Hence the above sequence
$\mathcal N$ is extended to a sutured manifold hierarchy. Now we
can apply \cite[Theorem~5.1]{G1} to the sutured manifold hierarchy
to obtain foliations as in the statement of Theorem~\ref{Gabai}.

Since $L$ is null-homotopic in $Y$, the distinguished slope that
kills $s(\delta_n)$ must be the meridian of $L$. As a result, the
sequence $\mathcal N$ for the slope $\alpha$ is taut. In
particular, the decomposition $$(G\times I=M_2(\alpha),\partial
G\times I)\stackrel{S}{\rightsquigarrow}(M_3(\alpha),\delta_3)$$
is taut.
\end{proof}

\begin{proof}[Proof of Theorem~\ref{FibredSurg} when $M$ is $F_T$--atoroidal]
By Proposition~\ref{DecomExt}, the condition in
Proposition~\ref{NoOrth} is satisfied, so the only possibility of
$G$ is that it is an annulus and
$$i_*\co\homo_1(K;\mathbb Q)\to\homo_1(G;\mathbb Q)$$ is an
isomorphism. Hence $\partial(G\times I)$ is I-cobordant to
$T=\partial\mathrm{Nd}(K)$. Since $M$ is $F_T$--atoroidal,
$\partial(G\times I)$ is parallel to $T$, so $K$ is isotopic to
the core of $G\times I$. This shows that $K$ can be isotoped to
lie on $F$.

Let $\lambda$ be the slope on $T$ which is specified by $F$. $F$
is compressible in $M(\lambda)$, hence not taut there.
Theorem~\ref{Gabai} then asserts that $\beta=\lambda$.
\end{proof}

Now consider the case when $M$ is not $F_T$--atoroidal, namely,
there exists a torus $R\subset M_1$ which is I-cobordant to $T$ in
$M_1$, but $R$ is not parallel to $T$.

Let us choose $R$ to be an ``innermost" torus in $M_1$ which is
I-cobordant to $T$ but not parallel to $T$. By
Lemma~\ref{BoundSolTor}, $R$ bounds a solid torus $U_X$ in $X$.
Any torus in $M-\mathrm{int}(U_X)$ which is I-cobordant to $R$ is
actually parallel to $R$.

\begin{proof}[Proof of Theorem~\ref{FibredSurg} when the above torus $R$ is present]
Let $K'$ be the core of $U_X$, $U_Y$ be the manifold obtained from
$U_X$ by $\beta$--surgery on $K$. By \cite{GSurg}, one of the
following cases must hold.

\noindent1) $U_Y=D^2\times S^1$. In this case $K$ is a braid in
$U_X$, and $L$ is a braid in $U_Y$.

\noindent2) $U_Y=U'\#W$, where $W$ is a closed $3$--manifold and
$1<|\homo_1(W)|<\infty$.

\noindent3) $U_Y$ is irreducible and $\partial U_Y$ is
incompressible.

In Case 1), let $L'$ be the core of $U_Y$, then $L'$ has finite
order in $\pi_1(Y)$. By Proposition~\ref{YFibred} $L'$ is
null-homotopic. $X$ can be viewed as obtained by a surgery on
$L'$. By the case we have proved, we have
$\Delta(\alpha',\beta')=1$, where $\alpha'$ is the meridian of
$U_X$, $\beta'$ is the meridian of $U_Y$.

Let $w,v$ be the winding numbers of $K,L$ in $U_X,U_Y$, then
$\alpha'$ cobounds a $w$--punctured disk $D_X$ with $w$ copies of
$\alpha$ in $U_X-\mathrm{Nd}(K)$, $\beta'$ cobounds a
$v$--punctured disk $D_Y$ with $v$ copies of $\beta$ in
$U_Y-\mathrm{Nd}(L)$. Consider the intersection of $D_X,D_Y$, we
get $\Delta(\alpha',\beta')=wv\Delta(\alpha,\beta)$, so $w=1$,
which means that $K$ is the core of $U_X$, a contradiction to the
assumption that $R$ is not parallel to $T$.

In Case 2), $Y$ would have a summand $W$, which contradicts to
Proposition~\ref{YFibred}.

In Case 3), $L$ is a null-homotopic knot in $U_Y$, then by
Proposition~\ref{DegOne} and \cite[Lemma~15.12]{Hem},
$\pi_1(U_X)=\mathbb Z$ surjects onto $\pi_1(U_Y)$, a contradiction
to the incompressibility of $\partial U_Y$.
\end{proof}

\section{Reducing the norm of fibred manifolds}

In this section, we are going to prove Theorem~\ref{NormDecr} The
notation is as in Theorem~\ref{NormDecr}.

\begin{lem}\label{DisjFibre2} $K$ can be isotoped in $X$ to be disjoint from
$F$.
\end{lem}
\begin{proof}
Otherwise, the Thurston norm of $\theta$ would be larger than
$x(F)$. Let $(F',\partial F')\subset(M,\partial M-T)$ be a taut
surface in the homology class $\theta$. By
\cite[Corollary~2.4]{G2}, $F'$ remains taut in at least one of $X$
and $Y$, which contradicts to the assumption that
$x(F')>x(F)>x(i_Y(\theta))$.
\end{proof}

Cut $X$ open along $F$, we get a product $F\times I$, let
$M_1,M_2,E,G$ be as in Section~\ref{SectProd}.

\begin{prop}\label{DecomExt2}
If $M$ is $F_T$--atoroidal, then $M_2$ has the taut-decomposition
property.
\end{prop}
\begin{proof}
The proof is the same as Proposition~\ref{DecomExt}, except that
we use the fact that $x(i_Y(\theta))<x(F)$ instead of the
null-homotopicity of $L$.
\end{proof}

Having Proposition~\ref{DecomExt2} in hand, the proof of
Theorem~\ref{NormDecr} when $M$ is $F_T$--atoroidal is the same as
 the proof of Theorem~\ref{FibredSurg}. Now let us consider the case
$M$ is not $F_T$--atoroidal, namely, there exists a torus
$R\subset M_1$ which is I-cobordant to $T$ in $M_1$, but $R$ is
not parallel to $T$.

Let us choose $R$ to be an ``innermost" torus in $M_1$ which is
I-cobordant to $T$. By Lemma~\ref{BoundSolTor}, $R$ bounds a solid
torus $U_X$ in $F\times I$. Any torus in $M-\mathrm{int}(U_X)$
which is I-cobordant to $R$ is actually parallel to $R$.

\begin{proof}[Proof of Theorem~\ref{NormDecr} when the above torus $R$ is present]
Let $K'$ be the core of $U_X$. $U_Y$ is the manifold obtained from
$U_X$ by $\beta$--surgery on $K$. Let $F_1$ be a norm minimizing
surface in the homology class $i_Y(\theta)$. $R$ is I-cobordant to
$T$, $\theta\in\homo_2(M,\partial M-T)$, so $F_1$ can be isotoped
to intersect $R$ in $2n$ essential circles, such that the sum of
these oriented circles is null-homologous in $R$. A standard
argument enables us to surger $F_1$ to get a new surface $F_2$ in
the same homology class, such that $\chi(F_2)=\chi(F_1)$, and
$F_2\cap R=\emptyset$.

If the slope of $F_1\cap R$ does not bound a disk in $U_Y$, then
$x(F_2)=x(F_1)$. Note that the components of $F_2$ in $U_Y$ are
null-homologous, removing these components we get a surface $F_3$
in the same homology class, $x(F_3)\le x(F_2)$. But $F_3\subset X$
is also a surface in the homology class of $[F]$ with $x(F_3)\le
x(F_2)=x(F_1)< x(F)$, we get a contradiction.

Hence the slope of $F_1\cap R$ bounds a disk in $U_Y$, which means
that $U_Y=U'\#W$, where $U'$ is a solid torus and $W$ is a
rational homology sphere by \cite{GSurg}. Let $\beta'$ be the
slope of $F_1\cap R$, $\alpha'$ be the meridian of $U_X$,
$M'=X-\mathrm{int}(U_X)$. Let $Y'$ be the manifold obtained from
$M'$ by $\beta'$--filling, then $Y=Y'\#W$. Since $W$ is a rational
homology sphere, the Thurston norm of $Y'$ is equal to the
Thurston norm of $Y$. Apply the $F$--atoroidal case of
Theorem~\ref{NormDecr} to $X,Y'$, we know that $K'$ can be
isotoped to lie on $F$, $\beta'$ is the slope specified by $F$.

Let $V=U_X-\mathrm{int}(\mathrm{Nd}(K))$. Suppose the winding
number of $K$ in $U_X$ is $w>0$, then $\homo_1(T)$ is generated by
$w[\beta']$ and $\frac1w[\alpha']$ in $\homo_1(V)$. Since $\beta'$
bounds a disk $D$ in $U_Y$, $[\beta']$ is an integer multiple of
$[\beta]$ in $\homo_1(V)$, namely,
$$[\beta']=k(rw[\beta']+s\frac1w[\alpha']), \quad k,r,s\in\mathbb Z.$$
 One then deduces that $w=1$, so $V$ is a homology $T^2\times I$,
 hence $\homo_1(U_Y)\cong\mathbb Z$ has no torsion. By
 \cite{GSurg}, the only possibility is $U_Y=D^2\times S^1$, hence $K$ is
 a braid in $U_X$. But $w=1$, so $T$ is parallel to $R$, a contradiction.
\end{proof}

\end{document}